\documentclass[a4paper,12pt]{amsart}
\usepackage{amsmath,amssymb,amsthm, stmaryrd, bm, bbm}
\usepackage{latexsym}
\usepackage{psfrag}
\usepackage{graphicx,subfigure}
\usepackage{graphicx}
\usepackage{dsfont}
\usepackage{tikz}
\usetikzlibrary{matrix,arrows,calc,snakes,patterns,decorations.markings}
\usepackage[mathscr]{eucal}
\usepackage{graphics,epic}
\usepackage{amsfonts}
\usepackage{amscd}
\usepackage{latexsym}
\usepackage[all,2cell]{xy}
\usepackage{mathrsfs}

\newtheorem{theorem}{Theorem}[section]
\newtheorem*{theorem*}{Theorem}

\newtheorem{lemma}[theorem]{Lemma}
\newtheorem{proposition}[theorem]{Proposition}
\newtheorem{corollary}[theorem]{Corollary}
\newtheorem*{corollary*}{Corollary}

\newtheorem*{conjecture*}{Conjecture}

\theoremstyle{remark}

\theoremstyle{definition}
\newtheorem{example}[theorem]{Example}
\newtheorem{remark}[theorem]{Remark}
\newtheorem{definition}[theorem]{Definition}

\newcommand{\ie}{{\em i.e.~}\ }

\newcommand{\eg}{{\em e.g.~}\ }

\newcommand{\opname}[1]{\operatorname{\mathsf{#1}}}

\renewcommand{\mod}{\opname{mod}\nolimits}

\newcommand{\proj}{\opname{proj}\nolimits}

\newcommand{\add}{\opname{add}\nolimits}

\newcommand{\thick}{\opname{thick}\nolimits}

\newcommand{\per}{\opname{per}\nolimits}

\newcommand{\id}{\mathrm{id}}

\newcommand{\Hom}{\opname{Hom}}
\newcommand{\End}{\opname{End}}

\newcommand{\ca}{{\mathcal A}}
\newcommand{\cb}{{\mathcal B}}
\newcommand{\cc}{{\mathcal C}}
\newcommand{\cd}{{\mathcal D}}

\newcommand{\ch}{{\mathcal H}}

\newcommand{\cp}{{\mathcal P}}

\newcommand{\ct}{{\mathcal T}}

\newcommand{\cx}{{\mathcal X}}

\newcommand{\udim}{\underline{\mathrm{dim}}}
\newcommand{\uDim}{\underline{\mathrm{Dim}}}
\newcommand{\Ind}{\mathrm{Ind}}
\renewcommand{\sp}{\mathrm{sp}}
\newcommand{\usum}{\underline{\mathrm{sum}}}
\newcommand{\uSum}{\underline{\mathrm{Sum}}}

\renewcommand{\leq}{\leqslant}
\renewcommand{\le}{\leqslant}
\renewcommand{\geq}{\geqslant}
\renewcommand{\ge}{\geqslant}

\numberwithin{equation}{section}

\begin{document}

\title[Derived-discreteness]{The equivalence of two discretenesses of triangulated categories}

\author{Lingling Yao}
\thanks{Lingling Yao acknowledges support from the National Natural Science Foundation of China under Grant No.11601077 and the
National Natural Science Foundation of Jiangsu Province BK20160662}
\address{
School of Mathematics, Southeast University, Nanjing
210096, China}
\email{llyao@seu.edu.cn}

\author{Dong Yang}
\address{Department of Mathematics, Nanjing University, Nanjing 210093, China}
\email{yangdong@nju.edu.cn}

\date{\today}
\maketitle

\begin{abstract}
Given an ST-triple $(\cc,\cd,M)$ one can associate a co-$t$-structure on $\cc$ and a $t$-structure on $\cd$. It is shown that the discreteness of $\cc$ with respect to the co-$t$-structure is equivalent to the discreteness of $\cd$ with respect to the $t$-structure. As a special case, the discreteness of $\cd^b(\mod A)$  in the sense of Vossieck is equivalent to the discreteness of $K^b(\proj A)$ in a dual sense, where $A$ is a finite-dimensional algebra.
\end{abstract}

\section{Introduction}

Derived-discreteness of a finite-dimensional algebra was introduced by Vossieck in~\cite{Vossieck01}. It is defined by counting the number of indecomposable objects in the bounded derived category.  Recently this notion has been generalised by Broomhead, Pauksztello and Ploog in~\cite{BroomheadPauksztelloPloog18} to a notion of discreteness of a triangulated category with respect to (the heart of) a bounded $t$-structure. In \cite{BroomheadPauksztelloPloog18} they also introduced a dual notion, namely, the notion of discreteness of a triangulated category with respect to a bounded co-$t$-structure (equivalently, a silting subcategory).

It turns out that ST-triples introduced in \cite{AdachiMizunoYang18} provide a nice framework to study the interplay between $t$-structures and co-$t$-structures. Let $\cc$ and $\cd$ be triangulated categories and $M$ a silting object of $\cc$ such that $(\cc,\cd,M)$ is an ST-triple. Then on $\cd$ there is a natural bounded $t$-structure, say, with heart $\ch$. Our main result is

\begin{theorem*}[\ref{thm:two-discretenesses}]\label{thm:main-thm}
The category $\cc$ is $M$-discrete if and only if the category $\cd$ is $\ch$-discrete.
\end{theorem*}

In the literature there are another two notions of discreteness of triangulated categories, namely, silting-discreteness \cite{Aihara13} and $t$-discreteness \cite{AdachiMizunoYang18}, defined by counting the number of silting objects and bounded $t$-structures, respectively. In \cite{AdachiMizunoYang18} it is shown that $\cc$ is silting-discrete if and only $\cd$ is $t$-discrete, and that if $\cd$ is $\ch$-discrete then $\cd$ is $t$-discrete. Together with these results Theorem~\ref{thm:two-discretenesses} implies the following corollary, which completes the picture.

\begin{corollary*}[{\ref{cor:discreteness-implies-silting-discreteness}}]
If $\cc$ is $M$-discrete, then $\cc$ is silting-discrete.
\end{corollary*}

The paper is structured as follows. In Section~\ref{s:prelinimaries} we fix the notion and briefly recall the definitions of $t$-structure, silting object and co-$t$-structure. In Section~\ref{s:ST-triple-and-discreteness} we recall the definitions of ST-triple and discreteness of triangulated categories. In Section~\ref{s:equivalence} we prove Theorem~\ref{thm:two-discretenesses}. In the final section we apply Theorem~\ref{thm:two-discretenesses} to finite-dimensional algebras to recover a result of Qin \cite{Qin16} stating that derived-discreteness in the sense of Vossieck \cite{Vossieck01} is preserved under decollement.

Throughout let $k$ be an algebraically closed field. We use $\Sigma$ to denote the shift functors of all triangulated categories.

\section{Preliminaries}\label{s:prelinimaries}
The aim of this section is mainly to briefly recall the definitions of $t$-structures, silting object and co-$t$-structure and fix the notation we will use in the paper.

\subsection{Triangulated categories}\label{ss:triangulated-categories}
Let $A$ be a finite-dimensional $k$-algebra. Denote by  $\mod A$ the category of finite-dimensional (right) $A$-modules and by $\proj A$ its full subcategory of finitely generated projective $A$-modules. Denote by $K^b(\proj A)$ the bounded homotopy category of $\proj A$ and by $\cd^b(\mod A)$ the bounded derived category of $\mod A$. These are two triangulated $k$-categories with shift functor being the shift of complexes.

\smallskip
Let $\ct$ be a triangulated $k$-category. For two subcategories $\ca$ and $\cb$, let $\ca\ast\cb$ be the full subcategory of $\ct$ consisting of objects $X$ with a triangle $X'\to X\to X''\to \Sigma X'$, where $X'\in\ca$ and $X''\in\cb$. We will often identify an object with the full subcategory consisting of this unique object. A full subcategory of $\ct$ is said to be \emph{thick} if it is closed under shifts, cones and direct summands. For an object $X$ of $\ct$ denote by $\add(X)$ the smallest additive subcategory of $\ct$ containing $X$ and closed under direct summands, and by $\thick(X)$ the smallest thick subcategory of $\ct$ containing $X$. Assume that $\ct$ has arbitrary coproducts. An object $X$ of $\ct$ is said to be \emph{compact} if the canonical map $\bigoplus_{i\in I}\Hom_\ct(X,Y_i)\to\Hom_\ct(X,\bigoplus_{i\in I}Y_i)$ is an isomorphism for any set $\{Y_i|i\in I\}$ of objects of $\ct$; it is called a \emph{compact generator} if in addition $\ct$ coincides with its smallest triangulated category containing $X$ and closed under coproducts.

\subsection{Grothendieck groups}
\label{ss:grothendieck-groups}

Let $\ch$ be an abelian $k$-category with only finitely many isoclasses (=isomorphism classes) of  simple objects  such that all objects of $\ch$ are filtered by simple objects (\eg $\mod A$, where $A$ is a finite-dimensional $k$-algebra). The Grothendieck group $K_0(\ch)$ of $\ch$ is the abelian group generated by isoclasses of objects in $\ch$ modulo the relations $[M]+[N]-[L]$ whenever there is a short exact sequence $0\to M\to L\to N\to 0$. For $M\in\ch$ denote by $\udim(M)$ the class of $M$ in $K_0(\ch)$. Let $K_0(\ch)^+$ be the subset of $K_0(\ch)$ consisting of classes of objects in $\ch$.  Then $K_0(\ch)$ is a free abelian group with basis the classes of simple objects, and in terms of this basis elements of $K_0(\ch)^+$ are precisely those with non-negative coefficients. 

\smallskip
Let $\ca$ be a Hom-finite Krull--Schmidt additive $k$-category such that $\ca=\add(M)$ for some $M\in \ca$ (\eg $\proj A$, where $A$ is a finite-dimensional $k$-algebra). The split Grothendieck group $K^{\sp}_{0}(\ca)$ of $\ca$ is the abelian group generated by the isoclasses of objects of $\ca$ modulo the relations $[L]+[N]-[L\oplus N]$. For $N\in \ca$, denote by $\usum(N)$ the class of $N$ in $K^{\sp}_{0}(\ca)$. Let $(K^{\sp}_{0}(\ca))^+$ be the subset of $K^{\sp}_0(\ca)$ consisting of classes of objects of $\ca$. Then $K^{\sp}_{0}(\ca)$ is a free abelian group with basis the classes of indecomposable direct summands of $M$, and in terms of this basis elements of ${K^{\sp}_{0}(\ca)}^+$ are precisely those with non-negative coefficients.

\subsection{$t$-structures}\label{ss:t-str}
Let $\ct$ be a triangulated $k$-category. 

A \emph{$t$-structure} on 
$\ct$ (\cite{BeilinsonBernsteinDeligne82}) is a pair $(\ct^{\leq
0},\ct^{\geq 0})$ of strict (that is, closed under
isomorphisms) and full  subcategories of $\ct$ such that, putting $\ct^{\geq p}=\Sigma^{-p}\ct^{\geq 0}$ and $\ct^{\leq p}=\Sigma^{-p}\ct^{\leq 0}$ for $p\in\mathbb{Z}$, we have
\begin{itemize}
\item[(1)] $\ct^{\leq -1}\subseteq\ct^{\leq 0}$ and
$\ct^{\geq 1}\subseteq\ct^{\geq 0}$;
\item[(2)] $\Hom(X,Y)=0$ for $X\in\ct^{\leq 0}$
and $Y\in\ct^{\geq 1}$,
\item[(3)] for each $X\in\ct$ there is a triangle $X'\rightarrow
X\rightarrow X''\rightarrow\Sigma X'$ in $\ct$ with $X'\in\ct^{\leq
0}$ and $X''\in\ct^{\geq 1}$.
\end{itemize}
The \emph{heart} $\ct^0:=\ct^{\leq 0}\cap\ct^{\geq 0}$ is
always abelian. The
$t$-structure $(\ct^{\leq 0},\ct^{\geq 0})$ is said to be
\emph{bounded} if
\[
\bigcup_{p\in\mathbb{Z}}\ct^{\leq
p}=\ct=\bigcup_{p\in\mathbb{Z}}\ct^{\geq p},
\] 
or equivalently, $\ct=\thick(\ct^0)$.

Let $A$ be a finite-dimensional $k$-algebra. Let $\cd^{\leq 0}$ (respectively, $\cd^{\geq 0}$) be the full subcategory of the bounded derived category $\cd^b(\mod A)$ consisting of complexes with
vanishing cohomologies in positive degrees (respectively, in negative
degrees). Then $(\cd^{\leq 0},\cd^{\geq 0})$ is a bounded $t$-structure on $\cd^b(\mod A)$ with heart the full subcategory of complexes with cohomology concentrated in degree $0$, which is canonically equivalent to $\mod A$.

\smallskip
It is easy to see that for every integer $p$, the pair $(\ct^{\le p},\ct^{\ge p})$ is also a $t$-structure and the category $\ct^{p}:=\ct^{\le p}\cap \ct^{\ge p}$ is the heart. By the condition (3) in the above definition, for $X\in\ct$ there is a triangle $X'\to X\to X''\to \Sigma X'$ with $X'\in\ct^{\leq p}$ and $X''\in\ct^{\geq p+1}$. This triangle is unique up to a unique isomorphism, so the correspondences $X\mapsto X'$ and $X\mapsto X''$ extend to functors 
\[
\sigma^{\le p}\colon  \ct \to \ct^{\le p}\ \textnormal{and}\ \sigma^{\ge p+1}\colon  \ct \to \ct^{\ge p+1},
\]
respectively, called the \emph{truncation functors}. Moreover, we have the set of \emph{cohomology functors} \[
\{\sigma^p=\Sigma^{p}\sigma^{\leq p}\sigma^{\geq p}\colon \ct\to \ct^0\mid p\in\mathbb{Z}\},
\] 
which is cohomological, \ie takes triangles to long exact sequences. The next result follows directly from the definition of $\sigma^{\geq p}$ on morphisms.

\begin{lemma}\label{lem:property-of-truncations-on-morphisms}
The map $\sigma^{\geq p}(X,Y)\colon \Hom_\ct(X,Y)\to\Hom_\ct(\sigma^{\geq p}X,\sigma^{\geq p}Y)$ 
\begin{itemize}
\item[$\cdot$] is injective if $\Hom_\ct(X,\sigma^{\leq p-1}Y)=0$;
\item[$\cdot$] is surjective if $\Hom_\ct(X,\Sigma\sigma^{\leq p-1}Y)=0$;
\item[$\cdot$] has kernel $\{f\colon X\to Y\mid f \text{ factors through the morphism } \sigma^{\leq p-1}(Y)\to Y\}$.
\end{itemize}
\end{lemma}

\subsection{Silting objects and co-$t$-structures}\label{ss:silting}

Let $\ct$ be a triangulated $k$-category. 

\smallskip
An object $M$ of $\ct$ is said to be \emph{presilting} if $\Hom_{\ct}(M,\Sigma^p M)=0$ for all positive integers $p$, and
\emph{silting} if in addition $\ct=\thick(M)$. See \cite{KellerVossieck88,AssemSoutoTrepode08,AiharaIyama12}. Let $A$ be a finite-dimensional $k$-algebra.
Then the free $A$-module $A_A$ of rank $1$ is a silting object of the bounded homotopy category $K^b(\proj A)$.

\smallskip

A \emph{co-$t$-structure} on
$\ct$ \cite[Definition 2.4]{Pauksztello08}  (or \emph{weight structure} in~\cite{Bondarko10}) is a pair $(\ct_{\geq 0},\ct_{\leq
0})$ of strict and full subcategories of $\ct$ such that, putting $\ct_{\geq p}=\Sigma^{-p}\ct_{\geq 0}$ and $\ct_{\leq p}=\Sigma^{-p}\ct_{\leq 0}$ for $p\in\mathbb{Z}$, we have
\begin{itemize}
\item[(0)] both $\ct_{\geq 0}$ and $\ct_{\leq
0}$ are additive and closed under taking direct summands;
\item[(1)] $\ct_{\geq 1}\subseteq\ct_{\geq 0}$ and
$\ct_{\leq -1}\subseteq\ct_{\leq 0}$;
\item[(2)] $\Hom(X,Y)=0$ for $X\in\ct_{\geq 1}$ and $Y\in \ct_{\leq 0}$;
\item[(3)] for each $X\in\ct$ there is a triangle $X'\rightarrow
X\rightarrow X''\rightarrow\Sigma X'$ in $\ct$ with $X'\in\ct_{\geq
1}$ and $X''\in\ct_{\leq 0}$.
\end{itemize}
The intersection $\ct_{\geq
0}\cap\ct_{\leq 0}$ is called the \emph{co-heart} of the co-$t$-structure $(\ct_{\geq 0},\ct_{\leq 0})$. 
A co-$t$-structure $(\ct^{\leq 0},\ct^{\geq
0})$ is said to be \emph{bounded}~\cite{Bondarko10} if
\[
\bigcup\nolimits_{p\in\mathbb{Z}}\ct_{\leq
p}=\ct=\bigcup\nolimits_{p\in\mathbb{Z}}\ct_{\geq p},
\]
or equivalently, $\ct=\thick(\ct_{\geq
0}\cap\ct_{\leq 0})$.

Let $A$ be a finite-dimensional $k$-algebra. Let $\cp_{\geq 0}$ (respectively, $\cp_{\leq 0}$) be the full subcategory of $K^b(\proj A)$ consisting of objects isomorphic to complexes with trivial components in negative degrees (respectively, in positive degrees). Then $(\cp_{\geq 0},\cp_{\leq 0})$ is a bounded co-$t$-structure on $K^b(\proj A)$ with co-heart $\add(A)$, which is canonically equivalent to $\proj A$.

\section{ST-triples and discreteness}
\label{s:ST-triple-and-discreteness}

In this section we recall the definition of ST-triple from \cite{AdachiMizunoYang18} and two notions of discreteness of triangulated categories from \cite{BroomheadPauksztelloPloog18}; moreover, we show that `compact silting objects' naturally produce ST-triples, and establish some auxiliary results which we will use in Section~\ref{s:equivalence}.

\subsection{ST-triples}
Let $\ct$ be a triangulated $k$-category. 

An \emph{ST-triple} inside $\ct$ \cite[Definition 4.3]{AdachiMizunoYang18} is a triple $(\cc,\cd,M)$, where $\cc$ and $\cd$ are thick subcategories of $\ct$ and $M$ is a silting object of $\cc$, such that 
\begin{itemize}
\item[(ST1)]  $\Hom_{\ct}(M,T)$ is finite-dimensional for any object $T$ of $\ct$, 
\item[(ST2)] $(\ct^{\le 0}, \ct^{\ge 0})$ is a $t$-structure on $\ct$, where for an integer $p$
\begin{eqnarray*}
\ct^{\le p}\hspace{-5pt}&:=\hspace{-5pt}&\{ X\in \ct \mid \Hom_{\ct}(M,\Sigma^m X)=0 ~\forall~m>p\},\\
\ct^{\ge p}\hspace{-5pt}&:=\hspace{-5pt}&\{ X\in \ct \mid \Hom_{\ct}(M,\Sigma^m X)=0 ~\forall~m<p\},
\end{eqnarray*}
\item[(ST3)] $\ct=\bigcup_{p\in\mathbb{Z}}\ct^{\leq p}$ and $\cd=\bigcup_{p\in\mathbb{Z}}\ct^{\geq p}$.
\end{itemize}

A prototypical example of an ST-triple is the triple $(K^b(\proj A),\cd^b(\mod A),A_A)$ inside $\cd^b(\mod A)$. Note, however, that in general $\cc$ and $\cd$ are not comparable, see \cite[the paragraph after Definition 4.3]{AdachiMizunoYang18}.

Let $\ca$ be a triangulated $k$-category with arbitrary coproducts. Assume that $M$ is a compact generator of $\ca$ such that $\Hom_\ca(M,\Sigma^p M)$ is finite-dimensional for all $p\in\mathbb{Z}$ and vanishes for all $p>0$. Put
\begin{align*}
\ca^c&=\thick(M),\\
\ca_{fd}&=\{X\in\ca\mid \bigoplus_{p\in\mathbb{Z}}\Hom_\ca(M,\Sigma^p X)\text{ is finite-dimensional}\},\\
\ca_{fd}^-&=\{X\in\ca\mid \Hom_\ca(M,\Sigma^p X)\text{ is finite-dimensional for all $p\in\mathbb{Z}$}\\ &\qquad \text{and vanishes for }p\gg 0\}.
\end{align*}
 All $\ca^c$, $\ca_{fd}$ and $\ca_{fd}^-$ are thick subcategories of $\ca$. By \cite[Theorem 3.4]{Keller06}, the category $\ca^c$ is precisely the subcategory of compact objects of $\ca$. Thus it is independent of the choice of $M$. It follows that $\ca_{fd}$ and $\ca_{fd}^-$ are independent of the choice of $M$ as well. Note that $\ca^c$ and $\ca_{fd}$ are contained in $\ca_{fd}^-$.

\begin{proposition}\label{prop:ST-triple-from-compact-silting}
Keep the notation and assumptions in the preceding paragraph.
Then 
\begin{itemize}
\item[(a)] both $\ca^c$ and $\ca_{fd}$ are Hom-finite and Krull--Schmidt,
\item[(b)] $(\ca^c,\ca_{fd},M)$ is an ST-triple inside $\ca_{fd}^-$.
\end{itemize}
\end{proposition}
\begin{proof}  (a) follows from (b) by Theorem~\ref{thm:basic-properties-of-ST-triple}(a) below.
It is clear that $M$ is a silting object of $\ca^c$. Let us prove (b) by verifying the three conditions in the definition of an ST-triple.

(ST1) This is true by the definition of $\ca_{fd}^-$.

(ST2) Put
\begin{align*}
\ca^{\leq 0}&=\{X\in\ca\mid \Hom_\ca(M,\Sigma^p X)=0 ~\forall~p>0\},\\
\ca^{\geq 0}&=\{X\in\ca\mid \Hom_\ca(M,\Sigma^p X)=0~\forall~p<0\}.
\end{align*}
Then by \cite[Theorem 1.3]{HoshinoKatoMiyachi02} (cf. also \cite[Proposition 2.8]{BeligiannisReiten07} and \cite[Corollary 4.7]{AiharaIyama12}), $(\ca^{\leq 0},\ca^{\geq 0})$ is a $t$-structure on $\ca$. For $X\in\ca$, consider the triangle 
\[
\xymatrix{X'\ar[r] & X\ar[r] & X''\ar[r] &\Sigma X'}
\]
with $X'\in\ca^{\leq 0}$ and $X''\in\ca^{\geq 1}$. Then by applying the functor $\Hom_\ca(M,?)$ to this triangle we obtain isomorphisms
\begin{align*}
\Hom_\ca(M,\Sigma^p X)&\cong \Hom_\ca(M,\Sigma^p X'') ~\forall p\geq 1,\\
\Hom_\ca(M,\Sigma^p X)&\cong \Hom_\ca(M,\Sigma^p X') ~\forall p\leq 0.
\end{align*}
As a consequence, if $X$ belongs to $\ca_{fd}^-$, so do $X'$ and $X''$. Therefore $(\ca_{fd}^{-,\leq 0},\ca_{fd}^{-,\geq 0})$ is a $t$-structure on $\ca_{fd}^-$, where $\ca_{fd}^{-,\leq 0}=\ca_{fd}^-\cap\ca^{\leq 0}$ and $\ca_{fd}^{-,\geq 0}=\ca_{fd}^-\cap\ca^{\geq 0}$ are the categories defined in the definition of an ST-triple.

(ST3) This is clear from the definitions of the involved categories.
\end{proof}

For a dg (=differential graded) $k$-algebra $A$,  it is known that the derived category $\cd(A)$ of dg $A$-modules (\cite{Keller94}) has arbitrary coproducts and is compactly generated by $A_A$, see \cite[Section 3.5]{Keller06}. Put
\begin{align*}
\per(A)&=\thick(A_A),\\
\cd_{fd}(A)&=\{X\in\cd(A)\mid \bigoplus_{p\in\mathbb{Z}}H^p(X)\text{ is finite-dimensional}\},\\
\cd^-_{fd}(A)&=\{X\in\cd(A) \mid H^p(X)\text{ is finite-dimensional for all $p\in\mathbb{Z}$}\\ &\qquad \text{and vanishes for }p\gg 0\}.
\end{align*}

\begin{proposition}[{\cite[Proposition 6.12]{AdachiMizunoYang18}}]
\label{prop:ST-pair-for-non-positive-dg-algebra}
Let $A$ be a dg $k$-algebra satisfying
\begin{itemize}
\item[(N)] $H^p(A)=0$ for any $p>0$,
\item[(F)] $H^p(A)$ is finite-dimensional for any $p\in\mathbb{Z}$.
\end{itemize}
Then 
\begin{itemize}
\item[(a)] both $\per(A)$ and $\cd_{fd}(A)$ are Hom-finite and Krull--Schmidt,
\item[(b)] $(\per(A),\cd_{fd}(A),A_A)$ is an ST-triple inside $\cd_{fd}^-(A)$.
\end{itemize}
\end{proposition}
\begin{proof}
This follows from Proposition~\ref{prop:ST-triple-from-compact-silting}, since $H^p(M)=\Hom_{\cd(A)}(A,\Sigma^p M)$ for a dg $A$-module $M$. 
\end{proof}

\smallskip
Let $(\cc,\cd,M)$ be an ST-triple inside $\ct$. Let $\ct_{\geq 0}$ be the smallest strict and full subcategory of $\ct$ which contains $M$ and is closed under extensions, direct summands and negative shifts, and let 
\[
\ct_{\leq 0}=\Sigma^{-1} \{X\in\ct\mid \Hom_\ct(Y,X)=0~\forall~ Y\in\ct_{\geq 0}\}.
\]
We collect some useful results in the following theorem.

\begin{theorem}\label{thm:basic-properties-of-ST-triple} 
Let $(\cc,\cd,M)$ be an ST-triple inside $\ct$.
\begin{itemize}
\item[(a)] {\rm (\cite[Remark 4.4(d)]{AdachiMizunoYang18})} Both $\cc$ and $\cd$ are Hom-finite and Krull--Schmidt.
\item[(b)] {\rm (\cite[Proposition 4.6(c)]{AdachiMizunoYang18})} $\bigcap_{p\in\mathbb{Z}}\ct^{\leq p}=0$.
\item[(c)] {\rm (\cite[Proposition 4.6]{AdachiMizunoYang18})} $(\cd^{\leq 0},\cd^{\geq 0}):=(\ct^{\leq 0}\cap \cd,\ct^{\geq 0})$ is a bounded $t$-structure on $\cd$ with heart $\cd^0=\ct^0$. The object $\sigma^0(M)$ is a projective generator of $\cd^0$, which is equivalent to $\mod \End_\ct(M)$.
\item[(d)] {\rm (\cite[Proposition 4.17]{AdachiMizunoYang18})} $(\ct_{\geq 0},\ct_{\leq 0})$ is a co-$t$-structure on $\ct$ with co-heart $\add(M)$ and $\ct_{\leq 0}=\ct^{\leq 0}$.
\item[(e)] {\rm (\cite[Remark 4.18]{AdachiMizunoYang18})} $(\cc_{\geq 0},\cc_{\leq 0}):=(\ct_{\geq 0},\ct_{\leq 0}\cap \cc)$ is a bounded co-$t$-structure on $\cc$ with co-heart $\add(M)$.
\end{itemize}
\end{theorem}

The following result is  \cite[Proposition 4.9]{AdachiMizunoYang18}. By (ST3), there exists $r\in\mathbb{Z}$ such that $Y\in\ct^{\leq r}$.

\begin{lemma}\label{lem:minimal-resolution}
Let $r\geq l$ be integers. For $Y\in\ct^{\leq r}$, there exist $\beta_{\geq l}(Y)\in\ct_{\geq l}$ and $\beta_{\leq l-1}(Y)\in\ct^{\leq l-1}$ and a triangle 
\[
\xymatrix{
\beta_{\geq l}(Y)\ar[r]^(0.6){f_Y} & Y\ar[r]^(0.35){g_Y} & \beta_{\leq l-1}(Y)\ar[r] & \Sigma\beta_{\geq l}(Y)
}
\]
with the following properties:
\begin{itemize}
\item[(a)] $\beta_{\geq l}(Y)\in \Sigma^{-r}M^r\ast\cdots\ast\Sigma^{-l}M^l$ for some $M^r,\ldots,M^l\in\add(M)$;
\item[(b)]  for any simple object $S$ in $\cd^0$ and  for all $p\geq l$ the map
\[
\xymatrix{
\Hom_\ct(Y,\Sigma^{-p} S)\ar[r]^(0.45){f_Y^*} &\Hom(\beta_{\geq l}(Y),\Sigma^{-p} S)
}
\]
is an isomorphism and the two spaces are isomorphic to $\Hom_\ct(M^{p},S)$;
\item[(c)]  for any simple object $S$ in $\cd^0$ and for all $p\leq l-1$ the map
\[
\xymatrix{
\Hom_\ct(\beta_{\leq l-1}(Y),\Sigma^{-p} S)\ar[r]^(0.57){g_Y^*} & \Hom_\ct(Y,\Sigma^{-p} S)
}
\]
is an isomorphism.
\end{itemize}
\end{lemma}

The objects $\beta_{\geq l}(Y)$ and $\beta_{\leq l-1}(Y)$ are constructed inductively. The first step goes as follows:  Take a minimal right $\add(\Sigma^{-r}M)$-approximation $f\colon \Sigma^{-r}M^r\to Y$ and form a triangle
\begin{eqnarray}
\xymatrix{
\Sigma^{-1}Y'\ar[r]^h &\Sigma^{-r}M^r\ar[r]^(0.6)f & Y\ar[r]^g & Y'.\label{eq:triangle-minimal-approx}
}
\end{eqnarray} 
Then $Y'\in\ct^{\leq r-1}$ because $M$ is silting. The `minimality' and uniqueness of $\beta_{\geq l}(Y)$ is established by inductively applying the following lemma. This is crucial in the definition of $\uSum$ in Section~\ref{sss:discreteness-wrt-silting}. The `limit' of $\beta_{\geq l}(Y)$ can be considered as a generalisation of minimal projective resolutions. 
Note, however, that in general $\beta_{\geq l}$ and $\beta_{\leq l-1}$ cannot be extended to functors.

\begin{lemma} 
Let $Y\in\ct^{\leq r}$.
Assume that there is a  triangle
\begin{eqnarray}
\xymatrix{
\Sigma^{-1}Y''\ar[r]^{h'} & \Sigma^{-r}N^r\ar[r]^(0.6){f'} & Y\ar[r]^{g'} & Y'',\label{eq:triangle-arbitrary-approx}
}
\end{eqnarray}
with $N^r\in\add(M)$. Then $f'$ is a right $\add(\Sigma^{-r}M)$-approximation if and only if $Y''\in\ct^{\leq r-1}$. If these conditions hold, then $h'$ is the direct sum of $h$ with an isomorphism in $\add(\Sigma^{-r}M)$ and \eqref{eq:triangle-arbitrary-approx} is the direct sum of \eqref{eq:triangle-minimal-approx} with a trivial triangle.
\end{lemma}
\begin{proof}
By inspection on the long exact sequence obtained by applying $\Hom_\ct(M,?)$ to \eqref{eq:triangle-arbitrary-approx} we obtain the first statement. If the conditions hold, then $f'$ is the direct sum of $f$ with a morphism $\Sigma^{-r}L^r\to 0$ with $L^r\in\add(M)$. The second statement follows.
\end{proof}

Repeatedly applying Lemma~\ref{lem:minimal-resolution}, we obtain the following corollary. 
\begin{corollary}\label{cor:minimal-resolution}
Let $Y\in\ct^{\leq r}$. Assume $N^r,\ldots,N^{l}\in\add(M)$ and $Y''\in\ct^{\leq l-1}$ with $Y\in\Sigma^{-r}N^r\ast\cdots\ast\Sigma^{-l}N^{l}\ast Y''$. Then $M^r,\ldots,M^{l}$ and $\beta_{\leq l-1}(Y)$ in Lemma~\ref{lem:minimal-resolution} are direct summands of $N^r,\ldots,N^{l}$ and $Y''\in\ct^{\leq l-1}$, respectively.
\end{corollary}

\subsection{Discretenesses}

Let $\ct$ be a triangulated $k$-category and $(\cc,\cd,M)$ an ST-triple inside $\ct$. We recall two notions of discreteness introduced in \cite{BroomheadPauksztelloPloog18}.

Recall that $M$ is a silting object of $\cc$ and on $\cd$ there is a bounded $t$-structure $(\cd^{\leq 0},\cd^{\geq 0})$ with heart $\cd^0$.

\subsubsection{Discreteness with respect to the $t$-structure}
\label{sss:discreteness-wrt-t-str}

For $X\in\cd$, define 
\[
\uDim(X)=(\udim \sigma^{p}(X))_{p\in \mathds{Z}}\in (K_{0}(\cd^0)^+)^{\oplus\mathds{Z}}.
\]

\begin{lemma}\label{lem:triangle-and-dimension-vector}
Let $X'\to X\to X''\to \Sigma X'$ be a triangle in $\ct$. Then
\[
\uDim(\sigma^{\geq p}(X))\leq \uDim(\sigma^{\geq p}(X'))+\uDim(\sigma^{\geq p}(X''))
\]
for any $p\in\mathbb{Z}$.
\end{lemma}
\begin{proof}
This is because $\{\sigma^p|p\in\mathbb{Z}\}$ is cohomological.
\end{proof}

For $x=(x^p)_{p\in\mathbb{Z}},y=(y^p)_{p\in\mathbb{Z}}\in K_0(\cd^0)^{\oplus \mathbb{Z}}$ define $y\leqslant x$ if $x^p-y^p\in K_0(\cd^0)^+$ for all $p\in\mathbb{Z}$. 
For $x\in K_{0}(\cd^0)^{\oplus\mathds{Z}}$, let $\Ind^x\cd$ (respectively, $\Ind^{\leqslant x}\cd$) be the isoclasses of indecomposable objects $X$ of $\cd$ with $\uDim(X)=x$ (respectively, $\uDim(X)\leqslant x$). 

\begin{definition}[{\cite[Definition 2.1]{BroomheadPauksztelloPloog18}}]
\label{defn:discreteness-wrt-t-structure}
The category $\cd$ is called \emph{$\cd^0$-discrete}  if the set $\Ind^x\cd$ is finite for any $x\in K_{0}(\cd^0)^{\oplus\mathds{Z}}$.
\end{definition}

\begin{example}
Let $A$ be a finite-dimensional $k$-algebra.
For $X\in\cd^b(\mod A)$, define $\uDim(X)=(\udim H^{i}(X))_{i\in \mathds{Z}}$, which belongs to the cone $(K_{0}(\mbox{mod}A)^+)^{\oplus\mathds{Z}}$. 
The algebra $A$ is called \emph{derived-discrete} \cite{Vossieck01} if the number of isoclasses of indecomposable objects $X$ of $\cd^b(\mod A)$ with $\uDim(X)=x$ is finite for any $x\in K_{0}(\mbox{mod}A)^{\oplus\mathds{Z}}$. It is clear that this is exactly the $(\mod A)$-discretenss in the sense of Definition~\ref{defn:discreteness-wrt-t-structure}.
There is a classification of derived-discrete algebras in \cite{Vossieck01} and a description of the AR quivers of $K^b(\proj)$ and $\cd^b(\mod)$ in \cite{BobinskiGeissSkowronski04} (see also \cite{KalckYang18}).
\end{example}

\begin{example}
Let $A=k[t]$ with $\mathrm{deg}(t)=-1$. We consider it as a dg algebra with trivial differential. Then $A$ satisfies the conditions (N) and (F) in Proposition~\ref{prop:ST-pair-for-non-positive-dg-algebra}. Take the ST-triple $(\cc,\cd,M)=(\per(A),\cd_{fd}(A),A_A)$ inside $\cd^-_{fd}(A)$. Then $\cd^0$ is the semisimple abelian category with a unique simple object $S=A/(t)$ (up to isomorphism).

According to \cite[Theorem 4.1(ii)]{KellerYangZhou09} (see also \cite[Lemma 8.8]{Joergensen04}), all indecomposable objects of $\cd_{fd}(A)$ are of the form $\Sigma^m A/(t^l)$ ($m\in\mathbb{Z}$ and $l\in\mathbb{N}$).  Put $x(m,l)=\sum_{p=m}^{m+l-1}\uDim(\Sigma^p S)$. Then $\uDim(\Sigma^m A/(t^l))=x(m,l)$. Therefore we have
\[
\# \Ind^x\cd=\begin{cases} 1 & \text{if } x=x(m,l) \text{ for some } m\in\mathbb{Z} \text{ and } l\in\mathbb{N},\\
0 & \text{otherwise}.
\end{cases}
\]
As a consequence, $\cd$ is $\cd^0$-discrete.
\end{example}

\begin{lemma}\label{lem:equivalent-condition-for-D-discreteness}
The category $\cd$ is $\cd^0$-discrete if and only if  the set $\Ind^{\leqslant x}\cd$ is finite for any $x\in K_{0}(\cd^0)^{\oplus \mathds{Z}}$.
\end{lemma}

\begin{proof} The ``if" part is obvious. The ``only if" part follows from the equality
\[
\Ind^{\leqslant x}\cd=\bigcup_{0\leqslant y\leqslant x} \Ind^{y}\cd
\]
and the fact that the number of $y$ satisfying $0\leqslant y\leqslant x$ is finite.
\end{proof}

\subsubsection{Discreteness with respect to the silting object}\label{sss:discreteness-wrt-silting}

Take $Y\in\cc$. Then there exists $l\in\mathbb{Z}$ such that $\beta_{\leq l-1}(Y)=0$. Take $M^r,\ldots,M^{l}\in\add(M)$ as in Lemma~\ref{lem:minimal-resolution} and put $M^p=0$ if $p>r$ or $p<l$. Define 
\[\uSum(Y)=(\usum(M^p))_{p\in\mathbb{Z}}\in (K^{\sp}_0(\add(M))^+)^{\oplus \mathbb{Z}},\]
where $\usum$ is defined in Section~\ref{ss:grothendieck-groups}.

\begin{lemma}\label{lem:triangles-and-sums}
Let $Y'\to Y\to Y''\to \Sigma Y'$ be a triangle in $\ct$. Then for any $l\in\mathbb{Z}$
\[
\uSum(\beta_{\geq l}(Y))\leq \uSum(\beta_{\geq l}(Y'))+\uSum(\beta_{\geq l}(Y'')).
\]
\end{lemma}
\begin{proof}
According to Lemma~\ref{lem:minimal-resolution} there exist $L^r,\ldots,L^l$ and $N^r,\ldots,N^l$ such that $Y'\in \Sigma^{-r}L^r\ast\cdots\ast\Sigma^{-l}L^l\ast\beta_{\leq l-1}(Y')$ and $Y''\in \Sigma^{-r}N^r\ast\cdots\ast\Sigma^{-l}N^l\ast\beta_{\leq l-1}(Y'')$. It follows by induction that 
\[
Y\in Y'\ast Y''\subseteq  \Sigma^{-r}(L^r\oplus N^r)\ast\cdots\ast\Sigma^{-l}(L^l\oplus N^l)\ast(\beta_{\leq l-1}(Y')*\beta_{\leq l-1}(Y'')),
\]
because $\beta_{\leq l-1}(Y')\ast\Sigma^{-p}N^p=\beta_{\leq l-1}(Y')\oplus\Sigma^{-p}N^p\subseteq \Sigma^{-p}N^p\ast \beta_{\leq l-1}(Y')$  and $L^p\ast N^p=L^p\oplus N^p$ for $l\leq p\leq r$. By Corollary~\ref{cor:minimal-resolution}, we obtain the desired result.
\end{proof}

For $u=(u^{p})_{p\in \mathds{Z}}$, $v=(v^{p})_{p\in \mathds{Z}}\in K_{0}^{\sp}(\add(M))^{\oplus\mathds{Z}}$, define $v\leqslant u$ if $u^{p}- v^{p}\in {K_0^{\sp}(\add(M))}^+$ for all $p\in \mathds{Z}$.
For $u\in K_0^{\sp}(\add(M))^{\oplus\mathbb{Z}}$, let $\Ind_u \cc$ (respectively, $\Ind_{\leqslant u} \cc$) be the set of isoclasses of indecomposable objects $Y$ of $\cc$ with $\uSum(Y)=u$ (respectively, $\uSum(Y)\leqslant u$).

\begin{definition}\label{defn:discreteness-with-respect-to-silting}
The category $\cc$ is called \emph{$M$-discrete} if the set $\Ind_u \cc$ is finite for any $u\in {K^{\sp}_{0}(\add(M))}^{\oplus \mathds{Z}}$.
\end{definition}

\begin{example}
Let $A=k[t]$ with $\mathrm{deg}(t)=-1$ and consider it as a dg algebra with trivial differential. Take the ST-triple $(\per(A),\cd_{fd}(A),A_A)$ inside $\cd^-_{fd}(A)$.

According to \cite[Theorem 4.1(i)]{KellerYangZhou09}, all indecomposable objects of $\per(A)$ are of the form $\Sigma^m A/(t^l)$ ($m\in\mathbb{Z}$ and $l\in\mathbb{N}$) or $\Sigma^m A_A$ ($m\in\mathbb{Z}$).  Put $u(m)=\uSum(\Sigma^m A_A)$ and put $u(m,l)=u(m)+u(m+l)$. Then $\uSum(\Sigma^m A/(t^l))=u(m,l)$. Therefore we have
\[
\# \Ind_u\per(A)=\begin{cases} 1 & \text{if } u=u(m,l) \text{ for some } m\in\mathbb{Z} \text{ and } l\in\mathbb{N},\\
1 & \text{if } u=u(m) \text{ for some } m\in\mathbb{Z},\\
0 & \text{otherwise}.
\end{cases}
\]
As a consequence, $\per(A)$ is $A_A$-discrete.
\end{example}

The following result is dual to Lemma~\ref{lem:equivalent-condition-for-D-discreteness} and its proof is similar. 

\begin{lemma}\label{lem:equivalent-condition-for-M-discreteness}
The category $\cc$ is $M$-discrete if and only if the set $\Ind_{\leqslant u} \cc$ is finite for any $u\in {K^{sp}_{0}(\add(M))}^{\oplus \mathds{Z}}$.
\end{lemma}


\begin{remark}
Using Lemma~\ref{lem:equivalent-condition-for-M-discreteness} one can show that $\cc$ is $M$-discrete if and only if it is discrete with respect to the co-$t$-structure $(\cc_{\geq 0},\cc_{\leq 0})$ in the sense of \cite[Definition 4.1]{BroomheadPauksztelloPloog18}. 
\end{remark}

\section{The two discretenesses are equivalent}\label{s:equivalence}

Let $\ct$ be a triangulated $k$-category and $(\cc,\cd,M)$ an ST-triple inside $\ct$. In Section~\ref{s:ST-triple-and-discreteness} we recalled two notions of discreteness in \cite{BroomheadPauksztelloPloog18}, one for $\cc$ and one for $\cd$. The following main result of this paper states that these two notions are equivalent. This has the flavour of Koszul duality.

\begin{theorem}\label{thm:two-discretenesses}
The category $\cc$ is $M$-discrete if and only if the category $\cd$ is $\cd^0$-discrete.
\end{theorem}

\begin{corollary}\label{cor:discreteness-implies-silting-discreteness}
If $\cc$ is $M$-discrete, then it is silting-discrete.
\end{corollary}
\begin{proof}
Assume that $\cc$ is $M$-discrete. Then $\cd$ is $\cd^0$-discrete by Theorem~\ref{thm:two-discretenesses}. The statement then follows from \cite[Theorems 7.9 and 7.1]{AdachiMizunoYang18}.
\end{proof}

We split Theorem~\ref{thm:two-discretenesses} into two propositions and prove them in Sections~\ref{ss:M-discreteness-implies-D-discreteness} and~\ref{ss:D-discreteness-implies-M-discreteness}, respectively.  In Section~\ref{ss:cone-finiteness} we discuss the relation between discreteness and cone-finiteness.

\smallskip
Recall that there is a triple $(\ct_{\geq 0},\ct_{\leq 0}=\ct^{\leq 0},\ct^{\geq 0})$, where $(\ct_{\geq 0},\ct_{\leq 0})$ is a co-$t$-structure and $(\ct^{\leq 0},\ct^{\geq 0})$ is a $t$-structure. The two proofs below are almost dual to each other. The subtle but serious difference comes from the fact that truncations associated to $t$-structures are functorial while truncations associated to co-$t$-structures are not. However, the interplay between these truncations is interesting and plays an important role in the proofs.

\subsection{$M$-discreteness implies $\cd^0$-discreteness}
\label{ss:M-discreteness-implies-D-discreteness}

The aim of this subsection is to prove the following implication.

\begin{proposition}\label{prop:M-discreteness-implies-D-discreteness}
If $\cc$ is $M$-discrete, then $\cd$ is $\cd^0$-discrete.
\end{proposition}

The following result is a direct consequence of Lemma~\ref{lem:property-of-truncations-on-morphisms}.

\begin{proposition}\label{prop:faithful-truncation}
Let $p\in\mathbb{Z}$. The functor $\sigma^{\geq p}\colon\ct\to\ct^{\geq p}$ restricts to a fully faithful functor
\[
\sigma^{\geq p}\colon \ct_{\geq p}\to\ct^{\geq p}.
\]
\end{proposition}
\begin{proof}
Take $X,Y\in\ct_{\geq p}$. Since $\sigma^{\leq p-1}(Y)$ and $\Sigma\sigma^{\leq p-1}(Y)$ belong to $\ct^{\leq p-1}$, we have $\Hom_\ct(X,\sigma^{\leq p-1}(Y))=0=\Hom_\ct(X,\Sigma\sigma^{\leq p-1}(Y))$. The desired result follows from Lemma~\ref{lem:property-of-truncations-on-morphisms}.
\end{proof}

We immediately obtain the following corollary, taking into account that $\ct_{\geq p}\supseteq\ct_{\geq l}$ for $p\leq l$.

\begin{corollary}\label{cor:faithful-truncation}
For $Y\in\ct_{\geq l}$, the following are equivalent:
\begin{itemize}
\item[(i)] $Y$ is indecomposable,
\item[(ii)] $\sigma^{\geq p}(Y)$ is indecomposable for some $p\leq l$, 
\item[(iii)] $\sigma^{\geq p}(Y)$ is indecomposable for all $p\leq l$.
\end{itemize}
Moreover, for $Y,Z\in\ct_{\geq l}$, the following are equivalent:
\begin{itemize}
\item[(1)] $Y\cong Z$,
\item[(2)] $\sigma^{\geq p}(Y)\cong\sigma^{\geq p}(Z)$ for some $p\leq l$,
\item[(3)] $\sigma^{\geq p}(Y)\cong\sigma^{\geq p}(Z)$ for all $p\leq l$.
\end{itemize}
\end{corollary}

For $l\in\mathbb{Z}$, consider the group homomorphism
\[
\varphi_l\colon K_0^{\sp}(\add(M))^{\oplus\mathbb{Z}}\longrightarrow K_0(\cd^0)^{\oplus\mathbb{Z}}
\]
defined by $\uSum(\Sigma^p N)\mapsto \uDim(\sigma^{\geq l}(\Sigma^p N))$ for $N\in\add(M)$ and $p\in\mathbb{Z}$. It restricts to a map 
\[
\varphi_l\colon (K_0^{\sp}(\add(M))^+)^{\oplus\mathbb{Z}}\longrightarrow (K_0(\cd^0)^+)^{\oplus\mathbb{Z}}.
\]

\begin{proof}[Proof of Proposition~\ref{prop:M-discreteness-implies-D-discreteness}]
Assume that $\cd$ is $\cd^0$-discrete.

Take $u\in(K_0^{\sp}(\add(M))^+)^{\oplus\mathbb{Z}}$. Then there exist $r,l\in\mathbb{Z}$ such that $u^p=0$ for $p<l$ and for $p>r$. Put $x=\varphi_l(u)\in (K_0(\cd^0)^+)^{\oplus\mathbb{Z}}$.

Let $Y\in\Ind_u\cc$, \ie $\uSum(Y)=u$. Then by Lemma~\ref{lem:minimal-resolution} there exist $M^r,\ldots,M^l\in\add(M)$ with $\usum(M^p)=u^p$ such that $Y\in\Sigma^{-r}M^r\ast\cdots\ast\Sigma^{-l}M^l$. By repeatedly applying Lemma~\ref{lem:triangle-and-dimension-vector}, we obtain
\[
\uDim(\sigma^{\geq l}(X))\leq\uDim(\sigma^{\geq l}(\Sigma^{-r}M^r))+\ldots+\uDim(\sigma^{\geq l}(\Sigma^{-l}M^l))=\varphi_l(u)=x.
\] 
Therefore by Corollary~\ref{cor:faithful-truncation}, there is an injective map
\[
\xymatrix@R=0.5pc{
\Ind_u\cc\ar[r] & \Ind^{\leq x}\cd.\\
X\ar@{|->}[r] & \sigma^{\geq l}(X)
}
\]
By Lemma~\ref{lem:equivalent-condition-for-D-discreteness}, the $\cd^0$-discreteness of $\cd$ implies that $\Ind^{\leq x}\cd$ is finite. It follows that $\Ind_u\cd$ is finite, as desired.
\end{proof}

\subsection{$\cd^0$-discreteness implies $M$-discreteness}
\label{ss:D-discreteness-implies-M-discreteness}

The aim of this subsection is to prove the following implication.

\begin{proposition}\label{prop:D-discreteness-implies-M-discreteness}
If $\cd$ is $\cd^0$-discrete, then $\cc$ is $M$-discrete.
\end{proposition}

The key point of our proof is the following result, which, specialising to the ST-triple $(K^b(\proj A),\cd^b(\mod A),A_A)$, strengthens \cite[Proposition 2]{HanZhang16}.

\begin{proposition}\label{prop:hard-truncation-preserves-indecomposability}
Let $l\in\mathbb{Z}$. For $X\in\cd^{\geq l}$, there is a surjective algebra homomorphism
\[
\End_\ct(\beta_{\geq l-1}(X))\longrightarrow \End_\ct(X),
\]
whose kernel is contained in the radical of $\End_\ct(\beta_{\geq l-1}(X))$. As a consequence, the following are equivalent:
\begin{itemize}
\item[(i)] $X$ is indecomposable,
\item[(ii)] $\beta_{\geq p}(X)$ is indecomposable for some $p\leq l-1$,
\item[(iii)] $\beta_{\geq p}(X)$ is indecomposable for all $p\leq l-1$. 
\end{itemize}
Moreover, if $X,Y\in\cd^{\geq l}$ satisfy $\beta_{\geq l-1}(X)\cong\beta_{\geq l-1}(Y)$, then $X\cong Y$.
\end{proposition}
\begin{proof} Rotate the triangle in Lemma~\ref{lem:minimal-resolution}, we obtain a triangle
\[
\xymatrix{
\Sigma^{-1}\beta_{\leq l-2}(X)\ar[r]^(0.55){h_X} & \beta_{\geq l-1}(X)\ar[r]^(0.6){f_X} & X\ar[r]^(0.35){g_X} & \beta_{\leq l-2}(X).
}
\]
Since $\Sigma^{-1}\beta_{\leq l-2}(X)\in\ct^{\leq l-1}$ and $X\in\cd^{\geq l}=\ct^{\geq l}$,  this is the canonical triangle of $\beta_{\geq l-1}(X)$ associated to the $t$-structure $(\ct^{\leq l-1},\ct^{\geq l-1})$. In particular, $X\cong\sigma^{\geq l}\beta_{\geq l-1}(X)$ and $\Sigma^{-1}\beta_{\leq l-2}(X)\cong\sigma^{\leq l-1}\beta_{\geq l-1}(X)$. The `Moreover' part follows immediately.

Consider the algebra homomorphism induced by the functor $\sigma^{\geq l}$
\[
\End_\ct(\beta_{\geq l-1}(X))\longrightarrow \End_\ct(\sigma^{\geq l}\beta_{\geq l-1}(X))\cong \End_\ct(X).
\]
By Lemma~\ref{lem:property-of-truncations-on-morphisms} this homomorphism is surjective, because $\beta_{\geq l-1}(X)\in\ct_{\geq l-1}$ and $\Sigma\sigma^{\leq l-1}\beta_{\geq l-1}(X)\in\ct^{\leq l-2}$. Moreover, the kernel of this map is 
\[
I:=\{a\colon \beta_{\geq l-1}(X)\to\beta_{\geq l-1}(X)\mid a \text{ factors through } h_X\}.
\]
If $a\in I$ and $S$ is a simple object of $\cd^0$, then $\Hom(a,\Sigma^p S)=0$ for all $p\leq -l+1$ because $\Hom(h_X,\Sigma^p S)=0$. Moreover, $\Hom_\ct(\beta_{\geq l-1}(X),\Sigma^p S)=0$ for all $p>-l+1$ because $\beta_{\geq l-1}(X)\in\ct_{\geq l-1}$. Therefore $\Hom(a,\Sigma^pS)=0$ for all $p\in\mathbb{Z}$. We claim that $a$ belongs to the radical. Otherwise, write $\beta_{\geq l-1}(X)=Y_1\oplus\ldots \oplus Y_s$ with $Y_1,\ldots,Y_s$ indecomposable. Then $a$ has a summand $\lambda\cdot \id_{Y_i}$ with $\lambda\in k^\times$ for some $i=1,\ldots,s$. Thus restricting $\lambda^{-1}a$ to $Y_i$ we obtain $\id_{Y_i}$. It follows that $\Hom(\id_{Y_i},\Sigma^p S)=0$ for all $p\in\mathbb{Z}$, which implies that $Y_i\in\bigcap_{p\in\mathbb{Z}}\ct^{\leq p}=0$, a contradiction.
\end{proof}

For $l\in\mathbb{Z}$, consider the group homomorphism
\[
\psi_l\colon  K_0(\cd^0)^{\oplus\mathbb{Z}}\longrightarrow K_0^{\sp}(\add(M))^{\oplus\mathbb{Z}}
\]
defined by $\uDim(\Sigma^p S)\mapsto \uSum(\beta_{\geq l-1}(\Sigma^p S))$ for any simple object $S$ of $\cd^0$ and any $p\in\mathbb{Z}$. It restricts to 
\[
\psi_l\colon (K_0(\cd^0)^+)^{\oplus\mathbb{Z}} \longrightarrow (K_0^{\sp}(\add(M))^+)^{\oplus\mathbb{Z}} .
\]

\begin{proof}[Proof of Proposition~\ref{prop:D-discreteness-implies-M-discreteness}]
Assume that $\cc$ is $M$-discrete.

Take $x\in (K_{0}(\cd^0)^+)^{\oplus\mathds{Z}}$.  Let $l$ be the maximal integer such that $x^p=0$ for all $p<l$ and put $u=\psi_l(x)$. We will define a map 
\[
h\colon \Ind^x\cd\longrightarrow \Ind_{\leqslant u} \cc
\]
and show that it is injective. By Lemma~\ref{lem:equivalent-condition-for-M-discreteness}, the $M$-discreteness of $\cc$ implies that $\Ind_{\leqslant u} \cc$ is finite. It follows that $\Ind^x\cd$ is finite, as desired.

\smallskip

\noindent{\bf Step 1:} \emph{The definition and injectivity of $h$.} Let $X\in\cd$ be indecomposable with $\uDim(X)=x$. Define $h(X)=\beta_{\geqslant l-1}(X)$. By Proposition~\ref{prop:hard-truncation-preserves-indecomposability}, $h(X)$ is indecomposable and $h$ is injective.

\smallskip
\noindent{\bf Step 2:} \emph{The well-definedness of $h$.} Let $X\in \cd$ be indecomposable with $\uDim(X)=x$. We show by induction on $x$ that $\uSum(\beta_{\geq l-1}(X))\leq u$. If $X$ is a shift of a simple object of $\cd^0$, the inequality holds by the definition of $\psi_l$. Otherwise, take a simple subobject $S$ of $\sigma^l(X)$, consider the composition
\[
\Sigma^{-l}S\to \Sigma^{-l}\sigma^l(X)\to X,
\]
and form a triangle
\[
\Sigma^{-l}S\to X\to X'\to \Sigma^{-l+1}S.
\]
It follows from the octahedron axiom that  $x=x'+x''$, where $x'=\uDim(X')$ and $x''=\uDim(\Sigma^{-l}S)$. Thus
\begin{align*}
\uSum(\beta_{\geq l-1}(X))&\leq \uSum(\beta_{\geq l-1}(\Sigma^{-l}S))+\uSum(\beta_{l-1}(X'))\\
&\leq \psi_l(x'')+\psi_l(x')=\psi_l(x)=u,
\end{align*}
where the first inequality follows from Lemma~\ref{lem:triangles-and-sums}, and the second one follows from induction hypothesis.
\end{proof}

\subsection{Cone-finiteness}
\label{ss:cone-finiteness}

Following \cite{BroomheadPauksztelloPloog18}, we say that a triangulated category is \emph{cone finite} if for any two objects $X$ and $Y$, the subcategory $X\ast Y$ has only finitely many isoclasses of objects. Note that this is a property that passes to subcategories.

\begin{corollary}\label{cor:cone-finiteness}
The following conditions are equivalent:
\begin{itemize}
\item[(i)] $\cc$ is $M$-discrete,
\item[(ii)] $\cc$ is cone finite,
\item[(iii)] $\cd$ is $\cd^0$-discrete,
\item[(iv)] $\cd$ is cone finite.
\end{itemize}
\end{corollary}
\begin{proof}
(i)$\Leftrightarrow$(ii) is \cite[Theorem 4.2]{BroomheadPauksztelloPloog18}, (i)$\Leftrightarrow$(iii) is Theorem~\ref{thm:two-discretenesses} and (iii)$\Rightarrow$(iv) is \cite[Theorem 2.5(iii)]{BroomheadPauksztelloPloog18}.

(iv)$\Rightarrow$(iii): Assume that $\cd$ is cone finite. We claim that for any $x\in K_0(\cd^0)$ the number of isoclasses of objects $X$ in $\cd^0$ with $\udim(X)=x$ is finite, \ie $\cd^0$ is abelian discrete in the sense of \cite[Section 2]{BroomheadPauksztelloPloog18}. Then it follows from \cite[Corollary 2.6]{BroomheadPauksztelloPloog18} that $\cd$ is $\cd^0$-discrete.

We prove the claim by induction on $x$. If $x$ is a standard basis element of $K_0(\cd^0)\cong\mathbb{Z}^n$, then $X$ must be simple and the claim is true. In general, take a simple subobject $S$ of $X$ and form the short exact sequence $0\to S\to X\to X'\to 0$. Then $x=\udim(S)+\udim(X')$. Moreover, the above short exact sequence yields a triangle $S\to X\to X'\to \Sigma S$ in $\cd$, and hence $X\in S\ast X'$. Thus all objects $X$ of $\cd^0$ with $\udim(X)=x$ belong to the subcategory $\cx=\bigcup S\ast X'$, where the union is over all isoclasses of simple objects $S$ of $\cd^0$ and all isoclasses of objects $X'$ with $\udim(X')=x-\udim(S)$. By induction hypothesis, this is a finite union. Since $\cd$ is cone finite, each $S\ast X'$ has finitely many isoclasses of objects. It follows that $\cx$ has only finitely many isoclasses of objects and the claim is true.
\end{proof}

Corollary~\ref{cor:cone-finiteness} shows the validity of \cite[Conjecture 2.7(iv)]{BroomheadPauksztelloPloog18} in our setting.

\section{Derived-discreteness along decollements}

In this section we recall the notion of derived-discreteness of a finite-dimensional algebra due to Vossieck \cite{Vossieck01}, and apply Theorem~\ref{thm:two-discretenesses} to recover the following result due to Qin~\cite{Qin16}. For basics on recollements, we refer to \cite{AngeleriKoenigLiuYang17}.

\begin{proposition}[{\cite[Proposition 6]{Qin16}}]
\label{prop:derived-discreteness-along-decollement}
Let $A,B,C$ be finite-dimensional $k$-algebras and assume that there is a recollement of $\cd(A)$ by $\cd(B)$ and $\cd(C)$. If $A$ is derived-discrete, then so are $B$ and $C$.
\end{proposition}

\subsection{Derived-discreteness}\label{ss:der-discreteness}

Let $A$ be a finite-dimensional $k$-algebra.

For $X\in K^{b}(\proj A)$, take $Y$ minimal such that $Y\cong X$ in $K^{b}(\proj A)$. Define $\uSum(X)=(\usum(Y^{p}))_{p\in \mathds{Z}}$, which belongs to $(K^{\sp}_{0}(\proj A)^+)^{\oplus \mathds{Z}}$.
The algebra $A$ is called \emph{$K^{b}(\proj)$-discrete} if the number of isoclasses of indecomposable objects of $K^b(\proj A)$ with $\uSum(X)=u$ is finite for any $u\in {K^{\sp}_{0}(\proj A)}^{\oplus \mathds{Z}}$. It is easy to see that this is exactly the $A_A$-discreteness in the sense of Definition~\ref{defn:discreteness-with-respect-to-silting}.

Applying Corollary~\ref{cor:cone-finiteness} to the ST-triple $(K^b(\proj A),\cd^b(\mod A),A_A)$, we immediately obtain the following corollary which completes \cite[Corollary 4.4]{BroomheadPauksztelloPloog18}.

\begin{corollary}\label{cor:equivalences-of-der-discrentess-and-K-discreteness}
The following conditions are equivalent:
\begin{itemize}
\item[(i)] $A$ is $K^b(\proj)$-discrete,
\item[(ii)] $K^b(\proj A)$ is cone finite,
\item[(iii)] $A$ is derived-discrete,
\item[(iv)] $\cd^b(\mod A)$ is cone finite.
\end{itemize}
\end{corollary}

\subsection{Derived-discreteness is preserved along decollements}

In this subsection we prove Proposition~\ref{prop:derived-discreteness-along-decollement}.

\begin{proof}[Proof of Proposition~\ref{prop:derived-discreteness-along-decollement}]
Assume that $A$ is derived-discrete. By Corollary~\ref{cor:equivalences-of-der-discrentess-and-K-discreteness}, both $\cd^b(\mod A)$ and $K^b(\proj A)$ are cone finite. In the given recollement the middle left functor  restricts to a fully faithful triangle functor $\cd^b(\mod B)\to \cd^b(\mod A)$ and the upper right functor restricts to a fully faithful triangle functor $K^b(\proj C)\to K^b(\proj A)$. Hence both $\cd^b(\mod B)$ and $K^b(\proj C)$ are cone finite. By Corollary~\ref{cor:equivalences-of-der-discrentess-and-K-discreteness} again, both $B$ and $C$ are derived-discrete.
\end{proof}

In the rest of this subsection we give an alternative proof of Proposition~\ref{prop:derived-discreteness-along-decollement} using the equivalence Corollary~\ref{cor:equivalences-of-der-discrentess-and-K-discreteness}(i)$\Leftrightarrow$(iii) only. Note that using the full Corollary~\ref{cor:equivalences-of-der-discrentess-and-K-discreteness} both Lemma~\ref{lem:subcategories-inherit-der-discreteness} and Lemma~\ref{lem:subcategories-inherits-K-discreteness} are easy to obtain.

The alternative proof for $B$ being derived-discrete is the same as that in \cite{Qin16}, which relies on the following result appeared in the paragraph before \cite[Proposition 6]{Qin16}. The idea of the proof is the same as Vossieck's proof of the fact that derived-discreteness is preserved under derived equivalence (\cite[Proposition 1.1]{Vossieck01}).  Here we give full details. Let $A$ and $B$ be finite-dimensional $k$-algebras in the next two lemmas.

\begin{lemma}\label{lem:subcategories-inherit-der-discreteness}
Assume that  $F\colon \mathcal{D}^{b}(\mod B)\rightarrow \mathcal{D}^{b}(\mod A)$ is a fully faithful triangle functor. If $A$ is derived-discrete, so is $B$.
\end{lemma}
\begin{proof} The triangle functor $F$ induces a group homomorphism
\[
f\colon K_{0}(\mbox{mod}B)^{\bigoplus \mathds{Z}}\longrightarrow K_{0}(\mbox{mod}A)^{\bigoplus \mathds{Z}}
\]
such that $f(\underline{\mbox{Dim}}(\Sigma^p S_{i}^{B}))=\underline{\mbox{Dim}}(\Sigma^p F(S_{i}^{B}))$ for a complete set $\{S_i^B\}$ of simple $B$-modules and $p\in\mathbb{Z}$.

We claim that $\underline{\mbox{Dim}}F(X)\leqslant f(\underline{\mbox{Dim}}(X))$ for any $X\in \mathcal{D}^{b}(\mbox{mod}B)$. It follows that $F$ induces an injective map $\Ind^x \cd^b(\mod B)\to \Ind^{\leqslant f(x)}\cd^b(\mod A)$, which is a finite set due to the derived-discreteness of $A$ and Lemma~\ref{lem:equivalent-condition-for-D-discreteness}. Thus $B$ is derived-discrete.
 
We prove the claim by induction on $x:=\underline{\mbox{Dim}}(X)$. Recall from Step 2 of the proof of Proposition~\ref{prop:D-discreteness-implies-M-discreteness} that there is a triangle in $\cd^b(\mod A)$:
\[
\Sigma^{-1}X'\stackrel{g}{\longrightarrow} \Sigma^{-l}S\stackrel{f}{\longrightarrow} X\longrightarrow X'
\] 
such that $x=x'+x''$, where $x'=\uDim(X')$ and $x''=\uDim(\Sigma^{-l}S)$. By applying $F$ to this triangle and inspecting the associated long exact sequence of cohomologies, we obtain an inequality
\[
\uDim F(X)\leqslant \uDim F(\Sigma^{-l}S)+\uDim F(X').
\]
By induction hypothesis we have $\uDim F(X')\leqslant f(x')$. Since $\uDim F(\Sigma^{-l}S)= f(x'')$, it follows that $\uDim F(X)\leqslant f(x)$, as claimed.
\end{proof}

The following result is dual to Lemma~\ref{lem:subcategories-inherit-der-discreteness}.

\begin{lemma}\label{lem:subcategories-inherits-K-discreteness}
Assume that $G\colon K^{b}(\proj B)\rightarrow K^{b}(\proj A)$is a fully faithful triangle functor. If $A$ is $K^b(\proj)$-discrete, so is $B$.
\end{lemma}

\begin{proof} The triangle functor $G\colon K^{b}(\proj B)\longrightarrow K^{b}(\proj A)$ induces a group homomorphism
\[
g\colon K_{0}^{\sp}(\proj B)^{\oplus \mathbb{Z}}\longrightarrow K_{0}^{\sp}(\proj A)^{\oplus \mathbb{Z}}
\]
such that $g(\uSum(\Sigma^p P_{i}^{B}))=\uSum(\Sigma^p G(P_{i}^{B}))$ for a complete set $\{P_i^B\}$ of indecomposable projective $B$-modules and $p\in\mathbb{Z}$.

We claim that $\uSum G(X)\leqslant g(\uSum(X))$ for any $X\in K^{b}(\proj B)$. It follows that $G$ induces an injective map $\Ind_u K^b(\proj B)\to \Ind_{\leqslant g(u)}K^b(\proj A)$, which is a finite set due to the $K^b(\proj)$-discreteness of $A$ and Lemma~\ref{lem:equivalent-condition-for-M-discreteness}. Thus $B$ is $K^b(\proj)$-discrete.

We prove the claim by induction on $u:=\uSum(X)$.
We may assume that $X$ is minimal. Let $r$ be the minimal integer such that $X^i=0$ for all $i> r$ and take an indecomposable direct summand $P$ of $X^{r}$.Then $\Sigma^{-r} P$ is a subcomplex of $X$ and there is a triangle
\[
\Sigma^{-r} P\longrightarrow X\longrightarrow X'\longrightarrow \Sigma^{-r+1} P
\]
with $\underline{\mbox{Sum}}(X')=\underline{\mbox{Sum}}(X)-\underline{\mbox{Sum}}(\Sigma^{-r} P)$. Applying $G$ to this triangle yields a triangle in $K^{b}(\proj A)$
\[
G(\Sigma^{-r} P)\longrightarrow G(X)\longrightarrow G(X')\longrightarrow G(\Sigma^{-r+1} P).
\]
Therefore
\begin{align*}
\underline{\mbox{Sum}} G(X) &\leqslant \underline{\mbox{Sum}}(G(\Sigma^{-r} P))+\underline{\mbox{Sum}}(G(X'))\\
&\leqslant g(\underline{\mbox{Sum}}(\Sigma^{-r} P))+g(\underline{\mbox{Sum}}(X'))\\
&= g(\underline{\mbox{Sum}}(\Sigma^{-r} P)+\underline{\mbox{Sum}}(X'))\\
&=g(\underline{\mbox{Sum}}(X)).
\end{align*}
Here the first inequality follows from Lemma~\ref{lem:triangles-and-sums} and the second one is by induction hypothesis.
\end{proof}

\begin{proof}[Alternative proof of Proposition~\ref{prop:derived-discreteness-along-decollement}]
The middle left functor in the given recollement restricts to a fully faithful triangle functor $\cd^b(\mod B)\to \cd^b(\mod A)$. So by Lemma~\ref{lem:subcategories-inherit-der-discreteness}, $B$ is derived-discrete.

Similarly, there is a fully faithful triangle functor $K^b(\proj C)\to K^b(\proj A)$. By Corollary~\ref{cor:equivalences-of-der-discrentess-and-K-discreteness}(i)$\Leftrightarrow$(iii), $A$ is $K^b(\proj)$-discrete. It follows from Lemma~\ref{lem:subcategories-inherits-K-discreteness} that $C$ is $K^b(\proj)$-discrete. By Corollary~\ref{cor:equivalences-of-der-discrentess-and-K-discreteness}(i)$\Leftrightarrow$(iii) again, $C$ is derived-discrete.
\end{proof}


\def\cprime{$'$}
\providecommand{\bysame}{\leavevmode\hbox to3em{\hrulefill}\thinspace}
\providecommand{\MR}{\relax\ifhmode\unskip\space\fi MR }
\providecommand{\MRhref}[2]{%
  \href{http://www.ams.org/mathscinet-getitem?mr=#1}{#2}
}
\providecommand{\href}[2]{#2}

\end{document}